\documentclass[12pt]{article}
\usepackage[margin=1in]{geometry}
\usepackage{url}
\usepackage{graphicx}
\usepackage{subcaption}

\usepackage[T1]{fontenc}

\usepackage{amsmath}
\usepackage{amssymb}
\usepackage{amsthm}
\usepackage{mathrsfs}

\theoremstyle{plain}
\newtheorem{theorem}{Theorem}

\newtheorem{lemma}[theorem]{Lemma}

\theoremstyle{definition}

\newtheorem{example}[theorem]{Example}
\newtheorem{conjecture}[theorem]{Conjecture}
\newtheorem{problem}{Problem}

\theoremstyle{remark}
\newtheorem{remark}[theorem]{Remark}

\newcommand{\VTM}{\mbox{\hspace{.01in}{\tt vtm}}}
\newcommand{\Thue}{\mbox{\hspace{.01in}{\tt vtm}}}
\newcommand{\EXTRA}[1]{}

\newcommand{\NN}{\omega}

\DeclareMathOperator{\lcp}{lcp}

\title{Some further results on squarefree arithmetic
  progressions in infinite words}
\author{
James Currie and Narad Rampersad\thanks{The authors are supported by
  NSERC Discovery Grants 03901-2017 (Currie) and 418646-2012 (Rampersad).}\\
Department of Mathematics and Statistics\\
University of Winnipeg\\
\url{{j.currie, n.rampersad}@uwinnipeg.ca}\\
\and
Tero Harju\\
Department of Mathematics and Statistics\\
         University of~Turku, Finland\\
         \texttt{harju@utu.fi}
\and
Pascal Ochem\thanks{The author is supported by ANR project CoCoGro
  (ANR-16-CE40-0005).}\\
LIRMM - CNRS\\
University of Montpellier, France\\
\url{ochem@lirmm.fr}\\
}
\date{\today}

\begin{document}
\maketitle

\begin{abstract}
  In a recent paper, one of us posed three open problems concerning
  squarefree arithmetic progressions in infinite words.  In this note
  we solve these problems and prove some additional results.
\end{abstract}

\section{Introduction}
The study of infinite words avoiding squares is a classical problem in
combinatorics on words.  A \emph{square} is a word of the form $xx$,
such as \texttt{tartar}.  One of the most well-studied squarefree
words~\cite{Thu12, Hal64} is the word $\Thue = 012021012102012\cdots$
obtained by iterating the map $0 \mapsto 012$, $1 \mapsto 02$, $2
\mapsto 1$.

A number of authors, such as Carpi~\cite{Car88}, Currie and Simpson
\cite{CS02}, and Kao et al.~\cite{KRSS08}, have studied squarefree
arithmetic progressions in infinite words.  Let $p$ be a positive
integer.  If $w = w_0w_1w_2\cdots$ is an infinite word where each
$w_i$ is a single letter, then $[w]_p$ denotes the infinite word
$w_0w_{p}w_{2p}\cdots$.  Harju~\cite{Har18} studied the following
question and showed that it has a positive solution for all $p\geq3$:
\begin{quote}
  Given $p$, does there exist an infinite squarefree word
  $w$ over a ternary alphabet such that $[w]_p$
  is squarefree?
\end{quote}
At the end of his paper, Harju posed three open problems:
\begin{problem}\label{prob1}
Does there exist a squarefree word $w$ over
    a ternary alphabet such that for every $p\geq 3$, the subsequence
    $[w]_p$ contains a square?
\end{problem}

\begin{problem}\label{prob2}
Do there exist pairs $(p,q)$ of relatively prime integers
    such that there exists a squarefree word $w$
    over a ternary alphabet for which both $[w]_p$
    and $[w]_q$ are squarefree?
\end{problem}

\begin{problem}\label{prob3}
It is true that for all squarefree words $v$ over a ternary alphabet,
there exists a squarefree word $w$ and an integer $p\geq 3$ such that $[w]_p = v$?
\end{problem}

In this paper we answer each of these questions.  Regarding
Problem~\ref{prob3}, we resolve this problem by instead studying the
following much stronger version of this problem.

\begin{problem}\label{prob4}
Let $p\geq 2$ be an integer and let $v$ be \emph{any} infinite ternary word.
Does there exist an infinite ternary squarefree word $w$ such that $[w]_p = v$?
\end{problem}

For $6\leq p\leq 29$ and $v = 00000\cdots$, we give a positive
answer to Problem~\ref{prob4}.  For $p\geq 30$, we give a positive
answer for all infinite words $v$.  In fact, we prove an even stronger
result by showing that, given any sequence $(p_i)_{i\geq 0}$ of
positions such that $p_{i+1}-p_i\geq 30$, it is possible to construct
$w$ such that $v$ appears as the subsequence of $w$ indexed by $(p_i)_{i\geq 0}$.

\section{Preliminaries}
Let $A$ be a finite alphabet of letters.
For a word $w$ over $A$ (i.e., $w \in A^*$), let $|w|$ denote its length,
i.e., the number of occurrences of letters in $w$.
A word $u$ is a \emph{factor} of $w$, if $w=xuy$ where $x$ and/or $y$ may be empty.
If $x$ ($y$, resp.) is empty then $u$ is a \emph{prefix} (a \emph{suffix}, resp.)  of $w$.

A finite or infinite word $w$ over $A$ is \emph{squarefree}
if it does not have any factors of the form $u^2=uu$ for nonempty words $u \in A^*$.
A morphism $h\colon A^* \to A^*$ is said to be \emph{squarefree},
if it preserves squarefreeness of words, i.e., if $h(w)$ is squarefree for all squarefree words $w$.
A morphism $h\colon A^* \to A^*$ is \emph{uniform} if the images $h(a)$ have the same length:
$|h(a)|=n$ for all $a \in A$ and for some positive $n$ called the \emph{length} of~$h$.

An infinite word $w$ is a \emph{fixed point} of a morphism $h$ if $h(w)=w$. This happens
if $w$ begins with the letter $a$, and $w$ is obtained by iterating $h$ on
the first letter~$a$ of $w$: $h(a)=au$ and
$w=auh(u)h^2(u)\cdots$. In this case we denote the fixed point $w$ by~$h^\omega(a)$.

For the next algorithmically helpful result we refer to Crochemore~\cite{Crochemore}.

\begin{theorem}\label{Crochemore}
  Let $h\colon A^* \to A^*$ be a morphism.
  \begin{enumerate}
    \item If $|A|=3$ then $h$ is squarefree if and only if it preserves
      squarefreeness of all squarefree words of length five.
    \item If $h$ is uniform then $h$ is squarefree if and only if it
      preserves squarefreeness of all squarefree words of length three.
  \end{enumerate}
\end{theorem}

In the rest of this paper we work over the ternary alphabet $T = \{0,1,2\}$.
Let $\tau\colon T^* \to T^*$ be the morphism defined by
\begin{equation}\label{Hall-Thue}
\tau(0)=012, \  \tau(1)=02 \  \text{and} \ \tau(2)=1\,.
\end{equation}
The morphism $\tau$ is not squarefree since it does not
preserve squarefreeness of the words $010$  and $02120$.
The images of these are
$01(2020)12$ and $01(210210)12$, respectively, with the squares indicated. Nevertheless,
the  word obtained by iterating the morphism on $0$ gives an infinite squarefree word
as a limit, called the \emph{Thue word}:
\[
\Thue=012 02 1 012 1 02 012 \cdots \ (=\tau^\omega(0)).
\]
(Here we follow~\cite{BSCFR14} in using $\VTM$, for \emph{variant
  of the Thue--Morse word}, to denote this word.)
For the next basic result, see~\cite{Braunholtz,Hal64,Istrail,
MorseHedlund,Thu12} and~\cite{Lothaire}:

\begin{lemma}\label{lemma:Hall_no}
The Thue word $\Thue$ is squarefree  and it does not contain $010$ or $212$ as factors.
\end{lemma}

Blanchet-Sadri et al.~\cite[Theorem~3]{BSCFR14} proved the following:

\begin{lemma}\label{lemma:blanchet-sadri}
  For each odd $k$ and each factor
  $w$ of $\VTM$, the set of positions at which $w$ occurs in $\VTM$
  contains all congruence classes modulo $k$.
\end{lemma}

\section{Problem~\ref{prob1}}
In this section we show that the word $\Thue$ gives a positive
solution to Problem~\ref{prob1}.

\begin{theorem}
  For each $k\geq 2$ the word $[\Thue]_k$ contains
  either the square $00$ or the square $22$.
\end{theorem}

\begin{proof}
  The first part of the proof relies on the fact that $\VTM$ is a
  $2$-automatic sequence.  Berstel~\cite{Ber79} studied several
  different ways to generate the sequence $\VTM$; in particular, he
  showed that $\VTM$ is generated by the $2$-DFAO (\emph{deterministic
    finite automaton with output}) in Figure~\ref{vtm_aut}.  The
  automaton takes the binary representation of $n$ as input (reading
  from most significant digit to least significant digit), and if
  the computation ends in a state labeled $a$, the automaton outputs
  $a$, indicating that $v_n = a$.

  \begin{figure}[htb]
    \centering
    \includegraphics[scale=0.75]{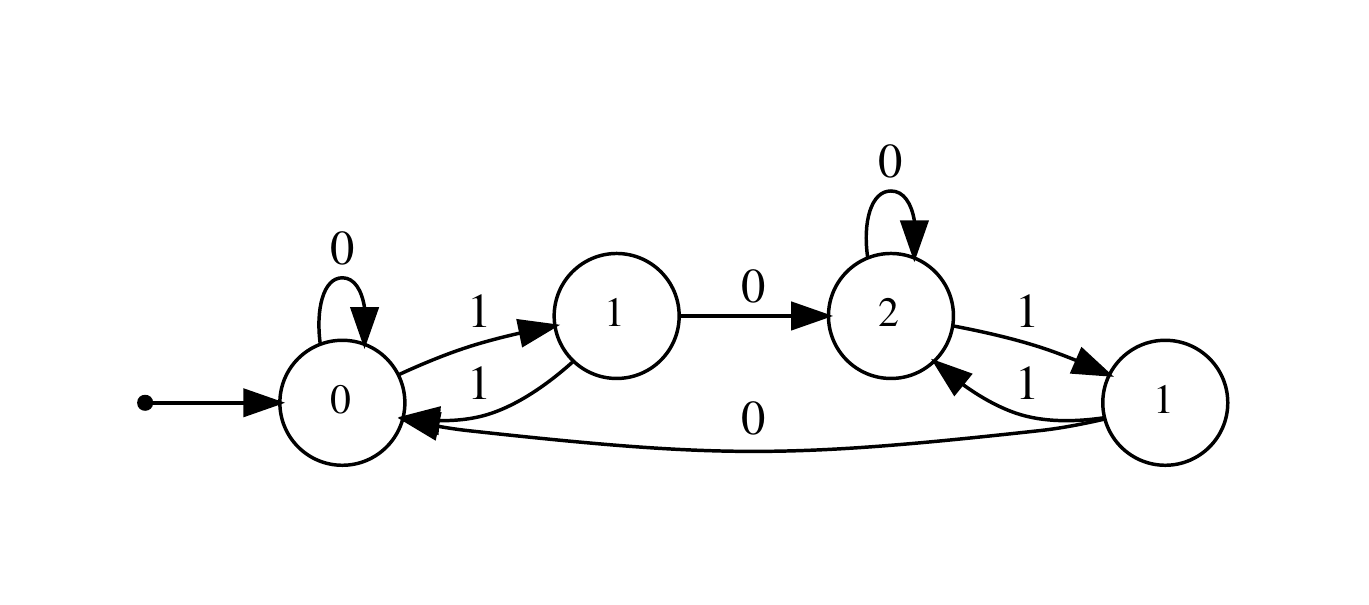}
    \caption{$2$-DFAO for $\VTM$}\label{vtm_aut}
  \end{figure}

  Since $\VTM$ is an automatic sequence, we can use
  Walnut~\cite{Walnut} to verify that it has certain combinatorial
  properties.  We verify with Walnut that for every $k\geq 2$, the
  sequence $\VTM$ contains a length $k+1$ factor of the form $0u0$ or
  $2u2$.  The Walnut command to do this is:
  \begin{center}
  \texttt{eval same\_first\_last ``Ei (VTM[i]=@0 \&
    VTM[i+k]=@0)|(VTM[i]=@2 \& VTM[i+k]=@2)'';}
  \end{center}
  The Walnut output for this command is the automaton in
  Figure~\ref{walnut_aut}, which shows that the given predicate holds
  for all $k\geq 2$.

  \begin{figure}[htb]
    \centering
    \includegraphics[scale=0.75]{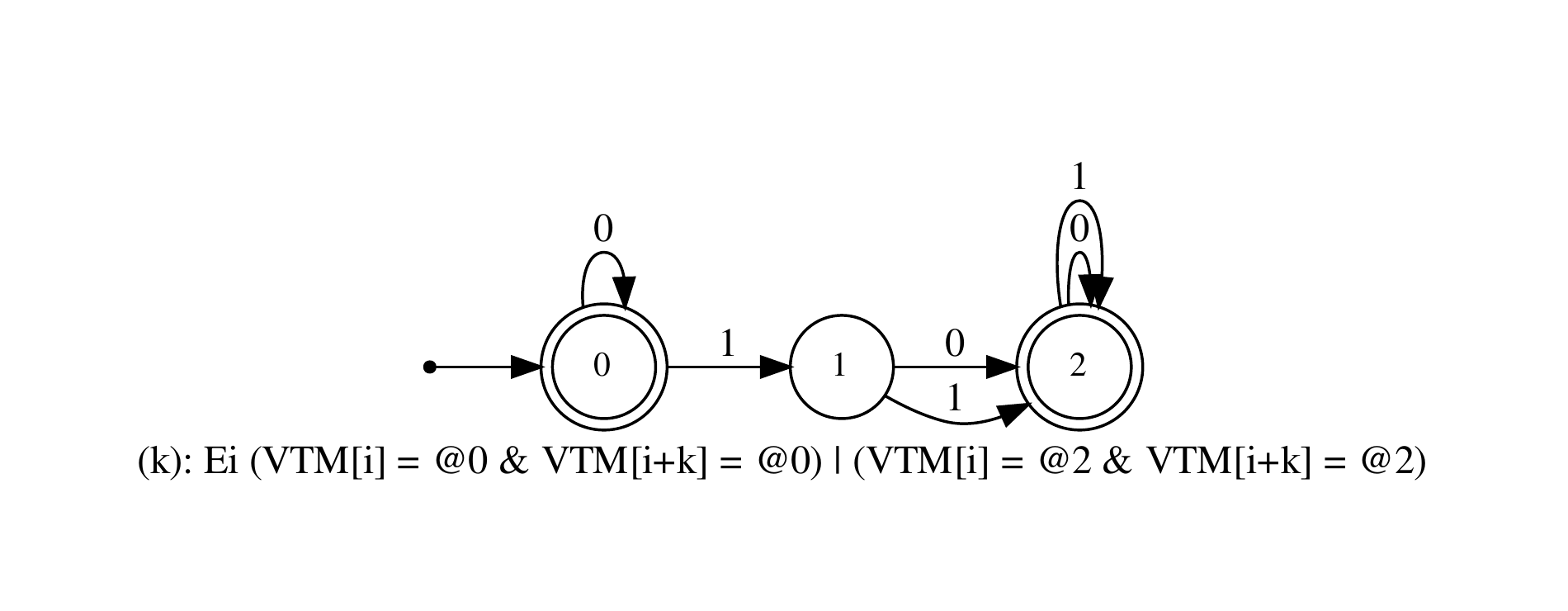}
    \caption{Walnut output automaton}\label{walnut_aut}
  \end{figure}

  Suppose $k$ is odd.  To complete the proof, it suffices to show that
  $\VTM$ contains an occurrence of the length $k+1$ factor $0u0$ or
  $2u2$ at a position congruent to $0$ modulo $k$.  This follows
  immediately from Lemma~\ref{lemma:blanchet-sadri} and so we have
  established the claim for all odd $k\geq 2$.

  If $k$ is even, write $k = 2^ak'$, where $k'$ is odd.  Suppose that
  $k'\geq 3$.  We have already seen that $\VTM$ contains an
  occurrence of a length $k'+1$ factor $0u0$ or $2u2$ at a position $i
  \equiv 0 \pmod{k'}$.  From the automaton generating $\VTM$, we see
  that if $v_i=0$ (resp.\ $v_i=2$), then $v_{2^ai}=0$
  (resp.\ $v_{2^ai}=2$), which establishes the claim for $k'\geq 3$.

  Finally, suppose that $k$ is a power of $2$.  Then the binary
  representations of $k$ and $2k$ have the form $10^\ell$ and
  $10^{\ell+1}$ respectively, for some $\ell\geq 1$.  From the
  automaton generating $\VTM$, we see that $v_k = v_{2k} = 2$, as
  required.  This completes the proof.
\end{proof}

\section{Problem~\ref{prob2}}
In this section we give positive solutions to Problem~\ref{prob2} for
two different pairs $(p,q)$; i.e.\ $(p,q)=(3,11)$ and $(p,q)=(5,6)$.
We first introduce some notation.  Given a morphism $h:A^* \to A^*$
and a positive integer $p$ such that $p$ divides $|h(a)|$ for all $a
\in A$, let $h_p$ be the morphism defined by $h_p(a) = [h(a)]_p$.

\begin{theorem}\label{pq_3_11}
  For $(p,q) = (3,11)$ there exists an infinite ternary squarefree
  word $w$ such that $[w]_p$ and $[w]_q$ are both squarefree.
\end{theorem}

We first define the following morphism $h$ on $T^*$:
\bigskip
{\small
\begin{align*}
0 \mapsto
&012101202102010212012101201021202101201020120210121020120212010210\\
&121021201020120212010210121020120212012102010210121020120212010210\\
&120212012102010212012101201021202102010212012101202102010210120212\\
1 \mapsto
&012101202102010212012101201021202101201020120210121020120212010210\\
&120212012102010212012101201021202102010212012101202102010210120212\\
2 \mapsto
&010210121021201020121021202101201021202102012101202102010210120212
\end{align*}
}

The next result gives a condition for a morphism $g$ to map $\VTM$ to a
squarefree word.  We will use it to show that $h(\VTM)$, $h_3(\VTM)$,
and $h_{11}(\VTM)$ are all squarefree, which is sufficient to
establish Theorem~\ref{pq_3_11}.

\begin{theorem}\label{img_of_vtm}
Suppose that $g:T^*\rightarrow A^*$ is such that
\begin{enumerate}
\item \label{squarefree} $g(u)$ is square-free for every factor $u$ of $\VTM$ of length $5$.
\item \label{010} The only solution of
\begin{eqnarray*}
g(a)&=&uv\\
g(b)&=&zv\\
g(c)&=&zw\\
a,b,c&\in& T,
\end{eqnarray*}
such that $v \neq w$ and $u \neq z$ is $abc=010$.
\item\label{suffix} For $a\in T$ we can write $g(a)=uv_a$ such that whenever $w\in T^*$ and $g(w)=xv_az$, then $z\in g(T^*)$.
\end{enumerate}
Then $g(\VTM)$ is square-free.
\end{theorem}

\begin{remark}\label{rem:h_h3_h11}
Let $h$, $h_3$, $h_{11}$ be given as above. Let $\hat{h}_{11}$ be defined on letters by $\hat{h}_{11}(a)=0^{-1}h_{11}(a)0$.
The theorem's conditions hold for $g=h,h_3,\hat{h}_{11}$.
\begin{itemize}
\item For $g=h$, condition~\ref{suffix} is witnessed by $v_a=02102010210120212$ for $a=0,1,2$.

\item For $g=h_3$, condition~\ref{suffix} is witnessed by $v_a=210212$ for $a=0,1,2$.

\item For $g=\hat{h}_{11}$, condition~\ref{suffix} is witnessed by $v_0=v_1=1210$, $v_2=20210$.
\end{itemize}
Thus, in order to establish Theorem~\ref{pq_3_11}, it remains to give the proof of Theorem~\ref{img_of_vtm}.
\end{remark}

\begin{proof}[Proof of Theorem~\ref{img_of_vtm}.]
Let $a_0a_1\cdots a_m$ be a shortest factor of $\VTM$ such that $g(a_0a_1\cdots a_m)$ contains a square.
Some non-empty $xx$ is a factor of $g(a_0a_1\cdots a_m)$. Since $m\geq 4$ by condition~\ref{squarefree},
and $m$ is as small as possible, it follows that $g(a_i)$ is a factor of $x$ for some $i$. 

Using condition~\ref{suffix} repeatedly, write
$$A_0^{\prime\prime}A_1\cdots A_{j-1}A'_j=A_j^{\prime\prime}A_{j+1}\cdots A_{2j-1}A'_{2j}$$
where 
\begin{eqnarray*}
A_i&=&g(a_i),0\leq i\leq 2j,\\
A_i&=&A_i'A_i^{\prime\prime}, i=0,j,2j\\
A_0^{\prime\prime}&=&A_j^{\prime\prime}\\
A_j'&=&A_{2j}'
\end{eqnarray*}
and $m=2j$.

By condition~\ref{squarefree}, $g$ is 1-1, so that $a_i=a_{j+i}$, $1\leq i\leq j-1$.
Since $\VTM$ is square-free, $A_0'\ne A_j'$ and $A_j''\ne A_{2j}''$. By condition~\ref{010}, this implies $a_0=a_{2j}=0$, $a_j=1$. 
Again,  since $\VTM$ is square-free, $a_0\ne a_1$, $a_{j-1}\ne a_j\ne a_{j+1}$, and $a_{2j-1}\ne a_{2j}$. 
Now $0=a_0\ne a_1=a_{j+1}\ne a_j=1$; thus $a_{j+1}=2$. Similarly, $0=a_{2j}\ne a_{2j-1}=a_{j-1}\ne a_j=1$;
thus $a_{j-1}=2$. Then $a_{j-1}a_ja_{j+1}=212$ is a factor of $\VTM$. This is a contradiction.
\end{proof}

\begin{remark}\label{rem:h11}
We also make the following observation concerning the morphism $h_{11}$, which is defined by
\begin{align*}
0 \mapsto &0201021\quad 012021\quad 20121\\
1 \mapsto &0201021\quad 20121\\
2 \mapsto &012021.
\end{align*}
By splitting the images in this way, we see that $h_{11}=x_{11}\circ\tau$, and hence that $h_{11}(\VTM)=x_{11}(\VTM)$,
where $x_{11}$ is the morphism defined by \[0 \mapsto 0201021,\quad 1 \mapsto 012021,\quad 2 \mapsto 20121.\]
One now verifies (using Theorem~\ref{Crochemore}) that $x_{11}$ is a
squarefree morphism, which implies that $h_{11}(\VTM)$ is squarefree.
\end{remark}

\begin{theorem}
  For $(p,q) = (5,6)$ there exists an infinite ternary squarefree
  word $w$ such that $[w]_p$ and $[w]_q$ are both squarefree.
\end{theorem}

\begin{proof}
To prove the result it suffices to find a morphism $h$ such that $h$, $h_5$, and $h_6$
are all squarefree (which we can verify using Theorem~\ref{Crochemore}). Such a morphism is given below.

\bigskip
{\small
\begin{align*}
0 \mapsto
&012102120210201021012021201210120210201021012102120121012021201020\\
&120210201021012010201210120212012102120210201202120102012101201021\\
&012102120121012021201020120210201021012010201210120212010201210212\\
&021012102010210120102012021012102120210201021201210120102012021020\\
&102120210120102012021201210212021012102012021201210120102101210201\\
&021202101201020120210121021202102010212012101201021202102010210121\\
&020120212012101201021202102010212012101201020120210201021202101210\\
&212012101201021202102010210121020120212012102010210121021202101201\\
&020120210201021202101210201202120121012010210121020102120121012021\\
&201020120210201021012010201210120212
\end{align*}
}
{\small
\begin{align*}
1 \mapsto
&012102120210201021012021201210120210201021012102120121012021201020\\
&120210201021012010201210120212012102120210201202120102012101201021\\
&012102120121012021201020120210201021012010201210120212010201210212\\
&021012102010210120102012021012102120210201021201210120102120210201\\
&021012102012021201210201021012102120210120102012021201210212021012\\
&102012021201210120102101210201021202101201020120210121021202102010\\
&212012101201021202102010210121020120212012101201021202102010212012\\
&101201020120210201021202101201020120212012102120210121020120212012\\
&101201021202102010210121020120212012102010210121021202101201020120\\
&210201021202101210201202120121012010210121020102120121012021201020\\
&120210201021012102120102012101201021012102120210201210120212
\end{align*}
}
{\small
\begin{align*}
2 \mapsto
&012102120210201021012021201210120210201021012102120121012021201020\\
&120210201021012010201210120212012102120210201202120102012101201021\\
&012102120121012021201020120210201021012010201210120212010201210212\\
&021012102010210120102012021012102120210201021201210120102120210201\\
&021012102012021201210201021012102120210120102012021201210212021012\\
&102012021201210120102101210201021202101201020120210121021202102010\\
&212012101201021202102010210121020120212012101201021202102010212012\\
&101201020120210201021202101210212012101201021202102010210121020120\\
&212012102010210121021202101201020120210201021202101210201202120121\\
&012010210121020102120121012021201020120210201021012010201210120212\\
&012102120210201202120102012101201021012102120210201210120212
\end{align*}
}
\end{proof}

\begin{remark}
  For the pair $(p,q)=(5,8)$ we searched via a backtracking algorithm for
  words of the desired form.  The backtrack appears to get ``stuck''
  around length $1200$, suggesting that there may not be an infinite
  word for the pair $(5,8)$.  This would be somewhat surprising, given
  the existence of an infinite word for the pair $(5,6)$.
\end{remark}

\section{Problems~\ref{prob3} and~\ref{prob4}}
Problem~\ref{prob3} is a special case of Problem~\ref{prob4}, so we
focus on Problem~\ref{prob4}.  We shall show that the least positive
answer for Problem~\ref{prob4} in the ternary case is $p=4$.  We also
state some positive and negative examples for small integers $p$ for
Problem~\ref{prob4}.

\subsection{Problem~\ref{prob4} for $2\leq p\leq 5$ and constant $v$}

\begin{lemma}
Problem~\ref{prob4} has no solution for the moduli $p=2$, $p=3$ and $p=5$.
\end{lemma}

\begin{proof}
Indeed, suppose first that $w \in T^\NN$ is an infinite
squarefree word such that $[w]_p= 00 \cdots$ is unary for $p \in \{2,3\}$.
Then necessarily $w \in \{01,02\}^\NN$ or $w \in \{012, 021\}^\NN$, respectively.
However, there are no infinite squarefree words in the binary case, and hence the moduli $2$ and $3$
have no solutions.

We then show that Problem~\ref{prob4} has no solution for the modulus $p=5$.
In this case if $[w]_5$ is constant then
\[
w \in  \{01021, 01201, 02102, 02012 \}^\NN.
\]
The longest words having an arithmetic progression of zeros modulo $p=5$
are the following three words of length 40:
\begin{align*}
&01021 01201 02012 02102 01021 01201 02012 02101\,,\\
&01021 01201 02012 02102 01021 01201 02012 02120\,,\\
&01201 02012 02102 01021 01201 02012 02102 01210\,,
\end{align*}
where there are `malapropos' factors $02120$ and $01210$ at the end of the two last words.
\end{proof}

\begin{remark}
If we allow four letters in the alphabet, then Problem~\ref{prob4}
becomes trivial for constant $v=0000\cdots$. Indeed,
let $p\geq 2$, and let
$w\in\{1,2,3\}^\NN$ be any squarefree ternary word. The quaternary word
obtained by adding 0 at the beginning and after every $p-1$ letters of~$w$
has an arithmetic progression of $0$'s modulo $p$.  For arbitrary $v$,
determining precisely the values of $p$ for which Problem~\ref{prob4}
has a solution over a four-letter alphabet remains an open problem.
\end{remark}

For modulus $p=4$ we have a solution for constant $v$.

\begin{theorem}\label{Hallplus}
The Thue word $\Thue$ has an infinite arithmetic progression of $1$'s modulo $4$.
\end{theorem}

\begin{proof}
We rely on the 3rd power of the morphism~$\tau$ defined in Eq.~\eqref{Hall-Thue}:
\begin{align*}
\tau^3(0)&=012021012102\\
\tau^3(1)&=01202102\\
\tau^3(2)&=0121.
\end{align*}
Since $\tau^3(\Thue)=\Thue$,
it is immediate that the letters of $\Thue$ at positions $n \equiv 1 \pmod 4$ are all equal to $1$.
\end{proof}

\subsection{Problem~\ref{prob4} for $6\leq p\leq 29$ and constant $v$}

\begin{theorem}\label{main2}
For all $p\geq 6$ divisible by an integer $m\leq 29$, there exists an
infinite squarefree word $w\in T^\NN$ such that $[w]_p=000\cdots$.
\end{theorem}

Before going to the constructions we state some aspects related to the problem.
For constant $v$, Theorem~\ref{main2} will settle the cases for $6\leq p\leq 29$.

\begin{remark}
Currie~\cite{Cur13} showed that there exist uniform squarefree morphisms
$T^* \to T^*$ for all lengths $n\geq 23$.
However, the constructed morphisms $h$ in~\cite{Cur13} are \emph{cyclic shift morphisms}
for which $h(1)= \pi(h(0))$ and $h(2)=\pi(h(1))$ where $\pi$ is the permutation
$(0\, 1\, 2)$ of $T$.
In particular the images of the letters start with different letters.  Because of this they do not
provide solutions to  our main problem; see Theorem~\ref{uniform} below.
\end{remark}

The following result narrows the candidates of infinite squarefree words
with arithmetic progressions obtained by morphisms. It also gives infinitely
many examples of integer sequences $a_1 < a_2 < \cdots$ avoiding infinite
arithmetic progressions with the property that the differences $a_{i+1}-a_i$
form a bounded sequence. Indeed, the bound is always $4$.
(Here $a_i$ will mark the place of the $i$th letter $0$ in the sequence
$h^\omega(0)$.)

\goodbreak
\begin{theorem}\label{uniform}
Let $h\colon T^* \to T^*$ be a uniform squarefree morphism of prime number length $q$ such that
\begin{equation}\label{cond}
\text{$0$ is a prefix  of $h(0)$ but not of $h(1)$ and $h(2)$.}
\end{equation}
Then there does not exist any infinite arithmetic progressions of $0$'s in $w= h^\omega(0)$.
\end{theorem}
\begin{proof}
Assume that there exists an integer $p$ such that $w$ has an arithmetic progression of $0$'s modulo~$p$.
Let $p$ be chosen to be the least of these integers.
Then the common length~$q$ of the images $h(a)$, $a \in T$, divides~$p$.
Indeed, otherwise $\gcd(p,q)=1$, as $q$ is a prime number,
and therefore the multiples of~$p$ form a complete residue system modulo~$q$.
In this case~$w$ is the all zero word; a contradiction.

Let $w= W_1W_2 \cdots$ where $|W_i|=p$. In particular, each $W_i$ begins with
the letter~$0$, and, by the above, for each $i$ there exists a unique word
$U_i \in T^*$ of length $p/q$ such that $h(U_i)=W_i$.  By~\eqref{cond}, each $W_i$ begins with $0$.
Now, since  $w$ is a fixed point of~$h$, it is also the case that $w=U_1U_2 \cdots$.
However, $|U_i|=p/q < p$ contradicts the minimality assumption on~$p$.
\end{proof}

\begin{example}
The least length for a uniform squarefree morphism is 11; see Brandenburg~\cite{Brandenburg}.
The following uniform morphism of length 11 is squarefree and satisfies  the conditions of
Theorem~\ref{uniform}:
\begin{align*}
h(0)&=02120102012\,,\\
h(1)&=10201021012\,,\\
h(2)&=10212021012\,.
\end{align*}
\end{example}

\begin{proof}[Proof of Theorem~\ref{main2}]
We proceed by constructing a uniform squarefree morphism~$h$ to
witness the claim for each $p$ with $6\leq p\leq 29$.  For this we need
only to give the morphisms for values of $p$ that are not divisible by
4 or by any $q$ with $6\leq q < p$.  The morphism $h$ is then applied
to any infinite squarefree word (such as the Thue word $\Thue$).  The
squarefreeness of the morphisms listed below can be checked by
Theorem~\ref{Crochemore} aided by a computer.

\bigskip

\noindent\phantom{x}
\begin{minipage}[t]{0.5\textwidth}
{\small\noindent
{\textbf{Case} $p=6$. Image length 36:}\\
$010212 012102 010212 021012 021201 021012$\,\\
$010212 012102 012021 012102 010212 021012$\,\\
$010212 012102 012021 020121 021201 021012$\,

}\end{minipage}\begin{minipage}[t]{0.5\textwidth}
{\small\noindent
{\textbf{Case} $p=7$.  Image length 28:}\\
$0121012 0210121 0201021 0120212$\\
$0121012 0210121 0212021 0120212$\\
$0121012 0210201 0212021 0120102$

}\end{minipage}

\bigskip

\noindent\phantom{x}
\begin{minipage}[t]{0.5\textwidth}
{\small\noindent
{\textbf{Case} $p=9$. Image length 18:}\\
$010210121 020120212$ \,\\
$010210121 020121012$ \,\\
$010210121 021202102$ \,

}\end{minipage}\begin{minipage}[t]{0.5\textwidth}
{\small\noindent
\text{\textbf{Case} $p=10$. Image length 20:}\\
$0102101201 0201202102$ \,\\
$0102101201 0201210212$ \,\\
$0102101201 0212012102$ \,

}
\end{minipage}

\bigskip

\noindent\phantom{x}
\begin{minipage}[t]{0.5\textwidth}
{\small\noindent
{\textbf{Case} $p=11$. Image length 11:}\\
$01021012102$\,\\
$01021202102$\,\\
$01210120212$\,

}\end{minipage}\begin{minipage}[t]{0.5\textwidth}
{\small\noindent

{\textbf{Case} $p=13$. Image length 26:}\\
$0120102012021 0120102120121$ \,\\
$0120102012021 0121021201021$ \,\\
$0120102012021 0201210212021$ \,

}
\end{minipage}

\bigskip

\noindent\phantom{x}
\begin{minipage}[t]{0.5\textwidth}
{\small\noindent
{\textbf{Case} $p=15$. Image length 30:}\\
$010201202101201 021012102010212$\\
$010201202101201 021201210120212$\\
$010201202101201 021202101210212$

}\end{minipage}\begin{minipage}[t]{0.5\textwidth}
{\small\noindent
\text{\textbf{Case} $p=17$. Image length 34:}\\
$01020120210201021 01201021201210212$\\
$01020121012021201 02101202102010212$\\
$01020121012021201 02101202120121012$

}
\end{minipage}

\bigskip

\noindent\phantom{x}
\begin{minipage}[t]{0.5\textwidth}
{\small\noindent
{\textbf{Case} $p=19$. Image length 19:}\\
$0102012021020121012$ \,\\
$0102012021201021012$ \,\\
$0102012102120121012$ \,

}\end{minipage}\begin{minipage}[t]{0.5\textwidth}
{\small\noindent
{\textbf{Case} $p=23$. Image length 23:}\\
$01020120210121020120212$ \,\\
$01020120210201210120212$ \,\\
$01020121012010210120212$ \,

}
\end{minipage}

\bigskip


\noindent\phantom{x}
\begin{minipage}{0.5\textwidth}
{\small\noindent
{\textbf{Case} $p=25$. Image length 25:}\\
$0102012021012010201210212$ \,\\
$0102012021012010210120212$ \,\\
$0102012021012102120121012$ \,

}
\end{minipage}\begin{minipage}{0.5\textwidth}
{\small\noindent
{\textbf{Case} $p=29$. Image length 29:}\\
$01020120210120102012021201021$ \\
$01210201202120102012102010212$ \\
$01210201202120102120210120212$
}
\end{minipage}

%
%
%
%
%
%
%
%
%
%
\end{proof}

We note that there are only cyclic shift morphisms for lengths $13$
and $17$, and there are no uniform morphisms of length $15$.  We also
note that the morphisms given above all have the property that the
images all have a common non-empty prefix.  This raises an interesting
question concerning the existence of $n$-uniform squarefree morphisms
with this property and the maximum length of the common prefix.  Given
a morphism $h:T^* \to T^*$, let $\lcp(h)$ denote the length of the
longest common prefix of the images $\{h(a) : a \in T\}$.  Given a
positive integer $n$, let $\lcp(n)$ denote the maximum value of
$\lcp(h)$ over all $n$-uniform squarefree morphisms $h:T^* \to T^*$.
We have computed the value of $\lcp(n)$ for $18\leq n\leq 70$.  The
results of this computation suggest the following conjecture:

\begin{conjecture}\label{conj:lcp}
For every $n\geq 30$, there exists an $n$-uniform squarefree morphism
$h:T^* \to T^*$ such that $\lcp(h) = n-6$.
\end{conjecture}

\begin{example}
  For $n=70$ the morphism $h$ with images
  {\small
  \begin{align*}
  h(0)&=0120210201021012010201202101201021012021020102120210120102012021201021\\
  h(1)&=0120210201021012010201202101201021012021020102120210120102012021020121\\
  h(2)&=0120210201021012010201202101201021012021020102120210120102012021012102
  \end{align*}
  }
  is squarefree and the images have a common prefix of length $64$.
\end{example}

We also note that the $n-6$ in Conjecture~\ref{conj:lcp} would be
optimal.  Using a computer we verified that for every triple of
distinct squarefree words $(u, v, w)$ with $|u|=|v|=|w|=5$, there do
not exist words $p$ and $s$ with $|p|=|s|=8$ such that $pus$, $pvs$,
and $pws$ are all squarefree.  It follows that the $n-6$ in the
conjecture cannot be replaced with $n-5$.

\subsection{Problem~\ref{prob4} for $p\geq 30$ and arbitrary $v$}

The next result implies a positive solution to Problem~\ref{prob4}
for $p\geq 30$ and arbitrary $v$, and also implies a positive
solution to Problem~\ref{prob3}.

\begin{theorem}
Let $(p_i)_{i\geq 0}$ be a sequence of positions such that
$p_{i+1}-p_i\geq 30$ and let $v$ be any infinite ternary word. There
exists an infinite squarefree ternary word $w$ such that $v$ appears
as the subsequence of $w$ indexed by $(p_i)_{i\geq 0}$.
\end{theorem}

\begin{proof}
We use the squarefree multi-valued morphism $h$ defined by
\begin{align*}
h(0)=&
\begin{cases}
012102120210\mathbf{12}021201210\\
012102120210\mathbf{201}021201210\\
012102120210\mathbf{2012}021201210\\
012102120210\mathbf{20121}021201210,
\end{cases}\\
\end{align*}
$h(1)=\pi(h(0))$, and $h(2)=\pi(h(1))$, where $\pi$ is the permutation
$(0\, 1\, 2)$ of $T$. To show that $h$ is squarefree, we apply
Theorem~\ref{Crochemore}(a).  It is easy to check that the proof of
Theorem~\ref{Crochemore} from \cite{Crochemore} works for multi-valued
morphisms as well; however, when applying the theorem, for each
squarefree word $x$ of length $5$ we have to check that all $4^5$ words
in the image $h(x)$ are squarefree.  Note that the images of each
letter have length $23$, $24$, $25$, or $26$. Also, they share a
common prefix of length $12$ and a common suffix of length $9$, so
that they only differ by their middle factors (in bold).

We start with any infinite ternary squarefree word, say the Thue word
$\Thue$. We apply $h$ to $\Thue$ with images of length $26$ only,
which gives an infinite ternary squarefree word $t = t_0t_1t_2\cdots$.  

Suppose that $t_{p_i} = v_i$; we will show how to modify $t$ so that
$t_{p_{i+1}} = v_{i+1}$.  If $t_{p_{i+1}} \neq v_{i+1}$, then $t$
contains the forced letter $v_{i+1}$ at position $p_{i+1}+1$,
$p_{i+1}+2$, or $p_{i+1}+3$.
Notice that every factor of $t$ of length $26+5-1=30$ contains a full occurrence of a middle factor of length $5$.
Since $p_{i+1}-p_i\geq 30$, there exist a full occurrence of a middle factor of length $5$
between positions $p_{i}+1$ and $p_{i+1}$.
If the match to $v_{i+1}$ is at position $p_{i+1}+1$ (resp.~$p_{i+1}+2$,~$p_{i+1}+3$),
then we can replace this middle factor by the middle factor of length $4$ (resp.~$3$,~$2$),
corresponding to the alternative $h$-image of length $25$ (resp.~$24$,~$23$).
This ensures that in the resulting word, the letter at position $p_{i+1}$ now matches $v_{i+1}$.
This procedure can be repeated for all $i$
and the resulting word is the desired word $w$.
\end{proof}

\section{Future work}
One obvious open problem is to characterize the pairs $(p,q)$ for which Problem~\ref{prob2} has a positive solution.
Another open problem is to prove Conjecture~\ref{conj:lcp}.
Regarding Problem~\ref{prob4}, we have only shown the existence of a positive solution for $p\geq 30$.
It remains to determine what happens for $p\leq 29$.
For example, for $p=4$, one can verify by computer that the infinite word $v=010101\cdots$ cannot be obtained
as a mod~$4$ subsequence of any squarefree ternary word.
Are there similar counterexamples for other small values of $p$, such as $p=6$?

\end{document}